\documentclass[reqno,twoside,12pt]{amsart}
\usepackage{amsmath}
 
\usepackage{amsfonts,amsmath,amssymb}
\usepackage{mathrsfs,mathtools}
\usepackage{enumerate}
\usepackage{hyperref}
\usepackage{esint}
\usepackage{graphicx}
\usepackage{bm}
\usepackage{commath}
\usepackage{esint}
\usepackage{enumitem}
\DeclareMathAlphabet{\mathpzc}{OT1}{pzc}{m}{it}

\usepackage[margin=1.5in]{geometry}

\newcommand{\innprd}[2]{\left( #1 , #2 \right)}
\newcommand{\op}[1]{{\mathcal #1}}
\newcommand{\R}{\mathbb{R}}

\newcommand{\embed}{\mathcal{E}^h}

\newtheorem{lemma}{Lemma}[section]
\newtheorem{conjecture}{Conjecture}[section]
\newtheorem{theorem}{Theorem}[section]

\begin{document}

\title[Multigrid for Fractional Optimization Problems]{A Note on Multigrid Preconditioning for Fractional PDE-Constrained Optimization Problems}

\author{Harbir Antil}
\address{H. Antil, Department of Mathematical Sciences and the Center for Mathematics and Artificial Intelligence (CMAI), George Mason University, Fairfax, VA 22030, USA.}
\email{hantil@gmu.edu}

\author{Andrei Dr{\u a}g{\u a}nescu}
\address{A. Dr{\u a}g{\u a}nescu, Department of Mathematics and Statistics, University of Maryland, Baltimore County, 
1000~Hilltop Circle, Baltimore, Maryland 21250, USA.}
\email{draga@umbc.edu}

\author{Kiefer Green}
\address{K. Green, Department of Mathematical Sciences and the Center for Mathematics and Artificial Intelligence (CMAI), George Mason University, Fairfax, VA 22030, USA.}
\email{kgreen32@gmu.edu}

\thanks{H. Antil and K. Green are  partially supported by NSF DMS-1913004 and DMS-1818772.  
A. Dr{\u a}g{\u a}nescu is partially supported by NSF DMS-1913201 and the U.S. Department of Energy
Office of Science, Office of Advanced  Scientific Computing Research under Award 
Number DE-SC0005455. }

\begin{abstract}
In this note  we present a multigrid preconditioning method for solving  quadratic 
optimization problems constrained by a fractional diffusion equation. Multigrid methods
within the all-at-once approach 
to solve the first order-order optimality Karush-Kuhn-Tucker (KKT) systems are widely popular, but
their development have relied on the underlying systems being sparse. On the other hand, for 
most discretizations, the matrix representation of fractional operators is expected 
to be dense. We develop a preconditioning strategy for our problem based on a reduced 
approach, namely we eliminate the state constraint 
using the control-to-state map. Our multigrid preconditioning approach shows a dramatic 
reduction in the number of CG iterations. We assess the quality of preconditioner in terms
of the spectral distance. Finally, we provide a partial theoretical analysis for this preconditioner, 
and we formulate a conjecture which is clearly supported by our numerical experiments.
\end{abstract}


\maketitle

\section{Introduction}

Let $\Omega \subset \mathbb{R}^N$ be an open bounded Lipschitz polygonal domain with boundary 
$\partial\Omega$. The goal of this paper is to develop an efficient multigrid based solver for the following 
optimal control problem: Given datum $u_d \in L^2(\Omega)$ and a regularization parameter $\beta >0$, 
{solve}
	\begin{subequations}\label{eq:ocp}
		\begin{equation}
			\min_{z \in L^2(\Omega)} \frac12 \|u-u_d\|^2_{L^2(\Omega)} + \frac{\beta}{2} \|z\|^2_{L^2(\Omega)} \, ,
		\end{equation}
		subject to the constraints posed by the fractional partial differential equation (PDE) 
		\begin{equation}\label{eq:state}
			\begin{cases}
				(-\Delta)^s u &= z \quad \mbox{in } \Omega \, , \\
						 u &= 0 \quad \mbox{on } \partial\Omega \, . 
			\end{cases}
		\end{equation}
	\end{subequations}
Here, $u$ and $z$ denote the state and control variables, respectively. Moreover, $(-\Delta)^s$, 
with $0 < s < 1$, denotes the $s$ powers of the $L^2(\Omega)$ realization of the Laplace's operator $-\Delta$, 
with the Dirichlet boundary condition $u = 0$ on $\partial\Omega$. This is the so-called spectral fractional
Laplacian. We refer to \cite{HAntil_JPfefferer_SRogovs_2018a} for the case of non-zero boundary conditions.
	
The rising interest of the community in fractional operators has been motivated by their ever-growing  applicability. 
In~\cite{CJWeiss_BGVBWaanders_HAntil_2020a} (see also \cite{HAntil_MDelia_CGlusa_CJWeiss_BGVBWaanders_2019a}
for an efficient solver), a fractional Helmholtz 
equation is derived using first principle arguments in-conjunction with a constitutive relationship. It also shows a direct 
qualitative match between numerical simulations and experimental data. In the classical setting, it is well-known that 
constrained optimization problems with the Helmholtz equation as constraint arise naturally in various applications. Examples
include  direct-field acoustic testing \cite{larkin2014developments} and remote sensing applications such as source inversion in
seismology \cite{auger2006real}. 
A natural first step to create efficient solvers for these optimization problems is to begin with optimization problems
constrained by Poisson type equations. 
Following this line of argument, we are hereby creating an efficient solver for \eqref{eq:ocp}.
Fractional operators have also received
a significant attention due to their applicability in imaging science~\cite{HAntil_SBartels_2017a,HAntil_ZDi_RKhatri_2020a}. 

Problem~\eqref{eq:ocp} was introduced in~\cite{HAntil_EOtarola_2014a}, and has attracted significant attention ever since.
While it is a natural extension of the standard elliptic control problem corresponding to the case $s=1$, it leads to a number
of challenging questions, beginning with the definition and the numerical representation of the fractional operator.
In~\cite{HAntil_EOtarola_2014a} problem~\eqref{eq:ocp} was formulated and analyzed using 
the extension approach~\cite{PRStinga_JLTorrea_2010a,LCaffarelli_LSilvestre_2007a}.
An alternative numerical analysis for~\eqref{eq:ocp} was provided in~\cite{SDohr_CKahle_SRogovs_PSwierczynski_2019a}. The latter 
used a numerical scheme to approximate~\eqref{eq:u}, based on Kato's formula~\cite{TKato_1960a}, orginally introduced 
in~\cite{ABonito_JEPasciak_2015a}. See also \cite{GHeidel_VKhoromskaia_BN_Khoromskij_VSchulz_2018a} for a tensor based method 
to solve \eqref{eq:ocp}. For completeness, we also refer to related optimal control problems corresponding to integral fractional Laplacian where the control is distributed \cite{HAntil_MWarma_2020a,MDElia_CGlusa_EOtarola_2019a}, or it is in the coefficient \cite{HAntil_MWarma_2019a,HAntil_MWarma_2018b}, or it is in the exterior \cite{HAntil_RKhatri_MWarma_2019a,HAntil_DVerma_MWarma_2020a}. We also refer to \cite{MAinsworth_CGlusa_2017a} for an efficient multigrid solver for fractional PDEs with integral fractional Laplacian.

The majority of efficient solution methods for solving PDE-constrained optimization problems focus on the first order optimality 
conditions, namely the Karush-Kuhn-Tucker (KKT) {system \cite{ABorzi_VSchultz_2012a}}. 
The KKT system couples the PDE~\eqref{eq:state} (the state equation) 
and the adjoint equation, the latter being a linear PDE with a similar character to the state equation. 
Hence, for the case of classical 
PDE constraints with finite element discretizations, 
the KKT system -- albeit indefinite -- will have have a sparse structure, and  solvers and preconditioners used for 
the state equation can play an important role for the KKT system as well. However, for most discretizations 
the matrix representation of discrete fractional operators is expected to be dense, therefore the all-at-once approach
of solving the KKT system loses its main {attractiveness}, namely sparsity.

In this work we use a reduced approach, namely we eliminate the state constraint from the optimization 
problem~\eqref{eq:ocp} using the control-to-state map.
Using the discretization from \cite{SDohr_CKahle_SRogovs_PSwierczynski_2019a}, we introduce 
a multigrid based preconditioner to solve \eqref{eq:ocp}. Multigrid methods, traditionally known as some of the most
efficient solvers of discretizations of PDEs, has been employed in recent times 
with great success in PDE-constrained optimization~\cite{ABorzi_VSchultz_2012a} as well.
Our approach is motivated by \cite{ADraganescu_TDupont_2008a}, and we 
develop a multigrid preconditioner for the reduced system of~\eqref{eq:ocp}. 

\section{The Fractional Operator and the Optimality Conditions}

\subsection{Continuous optimality conditions}
For $s \ge 0$, we define the fractional order Sobolev space
	\begin{equation}\label{eq:Hs}
		\mathbb{H}^s(\Omega) 
		:= \left\{ u = \sum_{k=1}^\infty u_k \varphi_k \in L^2(\Omega) \, : \, 
                \|u\|_{\mathbb{H}^s(\Omega)}^2 := \sum_{k=1}^\infty \lambda_k^s u_k^2 \ < \infty \right\} , 
	\end{equation}
where $\lambda_k$ are the eigenvalues of $-\Delta$ and $\varphi_k$ the corresponding eigenfunctions with zero Dirichlet boundary 
conditions and $\|\varphi_k\|_{L^2(\Omega)}=1$, and 
	\[
		u_k = (u,\varphi_k)_{L^2(\Omega)} = \int_\Omega u \varphi_k \, .
	\]	
By now, it is well-known that the definition of $\mathbb{H}^s(\Omega) $ in \eqref{eq:Hs} is equivalent to  
$H^s_0(\Omega)$ for $s > 1/2$, and 
$\mathbb{H}^s(\Omega) =H^s(\Omega) = H^s_0(\Omega)$ when $s < 1/2$, while $\mathbb{H}^{\frac12}(\Omega)
= H^{\frac12}_{00}(\Omega)$, i.e., the Lions-Magenes space \cite{LTartar_2007}. 
Recall that, {for $0 < s < 1$,}  $\mathbb{H}^s(\Omega)$ is the interpolation space between $L^2(\Omega)$ and $H^1_0(\Omega)$ \cite{LTartar_2007}, 
a fact that is relevant for the analysis {below}.
Let $\mathbb{H}^{-s}(\Omega)$ be the dual space of $\mathbb{H}^s(\Omega)$. 
	
For $s \ge 0$, the spectral fractional Laplacian is defined on the space $C_0^\infty(\Omega)$ by
	\[
		(-\Delta)^s u := \sum_{k=1}^\infty \lambda_k^s u_k \varphi_k  \quad \mbox{with} \quad u_k = \int_\Omega u\varphi_k \, .
	\]		
Notice that, for any $w = \sum_{k=1}^\infty w_k \varphi_k \in \mathbb{H}^s(\Omega)$, we have that 	
	\[
		\left|\int_\Omega (-\Delta)^s u w\right| = \left|\sum_{k=1}^\infty \lambda_k^s u_k w_k \right|
			= \left|\sum_{k=1}^\infty \lambda_k^{\frac{s}{2}} u_k \lambda_k^{\frac{s}{2}} w_k \right|
			\le \|u\|_{\mathbb{H}^s(\Omega)} \|w\|_{\mathbb{H}^s(\Omega)} ,
	\]
and thus $(-\Delta)^s$ extends as an operator mapping from $\mathbb{H}^s(\Omega)$ to $\mathbb{H}^{-s}(\Omega)$ due to 
density. In addition, we have that
	\[
		\|u\|_{\mathbb{H}^s(\Omega)} = \|(-\Delta)^\frac{s}{2} u\|_{L^2(\Omega)} \, .
	\]
Cf.~\cite{PRStinga_JLTorrea_2010a}, 
for every $z \in \mathbb{H}^{-s}(\Omega)$ there exists a unique $u \in \mathbb{H}^s(\Omega)$ that 
solves~\eqref{eq:state}. Using  Kato's formula (see \cite{ABonito_JEPasciak_2015a,HAntil_JPfefferer_2017a} for a derivation), the solution 
$u$ can be explicitly written as
	\begin{equation}
	\label{eq:u}
		u = (-\Delta)^{-s} z = \frac{\sin s\pi}{\pi} \int_{-\infty}^\infty e^{(1-s)y} (e^y-\Delta)^{-1} z \dif y \, .
	\end{equation}
Notice that $(-\Delta)^{-s} : \mathbb{H}^{-s}(\Omega) \rightarrow \mathbb{H}^{s}(\Omega)$ is bounded and linear. By
restricting $(-\Delta)^{-s}$ to $L^2(\Omega)$, and using the compact embedding $\mathbb{H}^s(\Omega)\hookrightarrow L^2(\Omega)$, 
we can treat the solution map $\mathcal{K}^s := (-\Delta)^{-s}$ as a bounded linear operator in $L^2(\Omega)$. 
Hence, the adjoint operator $(\mathcal{K}^s)^* : L^2(\Omega) \rightarrow L^2(\Omega)$ is well-defined, and is equal to $\mathcal{K}^s$. 
Using $\mathcal{K}^s$, the reduced form of problem \eqref{eq:ocp} is given by 
	\begin{equation}\label{eq:ocp_red}
		\min_{z \in L^2(\Omega)} \frac12 \|\mathcal{K}^sz-u_d\|^2_{L^2(\Omega)} + \frac{\beta}{2}\|z\|^2_{L^2(\Omega)} .
	\end{equation}
Problem \eqref{eq:ocp_red} has a unique solution that satisfies the following first-order necessary and sufficient 
optimality conditions
	\begin{equation}\label{eq:Ks}
		\mathcal{H}^s z \stackrel{\mathrm{def}}{=}\left((\mathcal{K}^s)^* \mathcal{K}^s+\beta I\right) z = (\mathcal{K}^s)^* u_d .
	\end{equation}
Notice that \eqref{eq:Ks} follows immediately after differentiating twice the functional in \eqref{eq:ocp_red}. 
The operator $\mathcal{H}^s$ in~\eqref{eq:Ks} is the continuous reduced Hessian operator. 	
Next we shall discretize~\eqref{eq:Ks}.

\subsection{Discrete optimality conditions}
%
We consider a quasi-uniform discretization $\op{T}_h$ of $\Omega$ and the spaces of 
continuous piecewise linear functions $\op{V}_h$ and 
$\op{V}^0_h=\op{V}_h\cap H_0^1(\Omega)$. The control $z$ is discretized using $\op{V}_h$,
while the state $u$ is discretized using $\op{V}^0_h$.
According to~\cite{ABonito_JEPasciak_2015a}, 
the discrete solution operator $\op{K}_h^s: \op{V}_h\to \op{V}^0_h$ is defined as
	\[
		\op{K}_h^s := \frac{\sin s\pi}{\pi} m \sum_{\ell = -N^{-}}^{N^+} e^{(1-s)y_\ell} 
		{(e^{y_\ell}-\Delta_h)^{-1}} \, ,
	\]
where the quadrature nodes are uniformly distributed as $y_\ell = m\ell$. This quadrature rule 
has been shown to be exponentially convergent (see \cite{ABonito_JEPasciak_2015a}) to the 
continuous integral in \eqref{eq:u}. The underlying constants $N^{-}$ and $N^{+}$ are chosen
to balance the quadrature error and spatial discretization error. In our case they are: 
{$m \sim (\ln \frac{1}{h})^{-1}$}, $N^+ = \lceil \frac{\pi^2}{4s m^2}\rceil$, and $N^- = \lceil \frac{\pi^2}{4(1-s) m^2}\rceil$. 
Finally, we shall denote by $\pi_h:L^2(\Omega)\to \op{V}_h$, the  $L^2$-orthogonal projection.

Using the above discretization, the discrete form of \eqref{eq:Ks} is given by 
\begin{eqnarray}
\label{equ:discretesys}
\op{H}_h^s z_h \stackrel{\mathrm{def}}{=}  ((\op{K}_h^s)^*\op{K}_h^s + \beta I)z_h= (\op{K}_h^s)^* u_{d,h} \, ,
\end{eqnarray}
where $u_{d,h}=\pi_h u_d$.
This work is concerned with a multigrid preconditioning approach to efficiently solve~\eqref{equ:discretesys}.

\section{Two-grid and multigrid preconditioner}
\label{sec:twogrid}

\subsection{Preconditioner description}
\label{ssec:twogriddescription}
Following~\cite{ADraganescu_TDupont_2008a}, assuming  $\op{T}_h$ is a refinement of $\op{T}_{2h}$, 
we define the {two-grid} preconditioner:
\begin{eqnarray}
\label{equ:twogridprec}
\op{G}_h^s = \beta (I-\embed\pi_{2h}) + \embed\op{H}_{2h}^s\pi_{2h} \, ,
\end{eqnarray}
where $\embed:\op{V}_{2h} \to \op{V}_h$ is the natural embedding operator. 

The extension of the preconditioners from two-grid to multigrid  is a streamlined process that 
is presented in full detail in~\cite{ADraganescu_TDupont_2008a, Barker_Draganescu_2020}. {It is sufficient} to say that the multigrid version has a 
W-cycle structure, and that the coarsest grid has to be sufficiently fine. Hence, it may be that the coarsest level used in the multigrid
version of $\op{G}_h^s$ is not the coarsest that is in principle available by the existing geometric framework. The number of levels that can
be used is problem dependent, and depends also of the quality of the two-grid preconditioner, as described below.

\subsection{Analysis and conjecture}
\label{ssec:twogridanalysis}
We assess the quality of the preconditioner $\op{G}_h^s$ by estimating the spectral distance (see~\cite{ADraganescu_TDupont_2008a})
$d(\op{H}_h^s,\op{G}_h^s)$, where for two symmetric positive definite operators $A,B\in \mathfrak{L}(L^2(\op{V}_h))$ 
\begin{eqnarray}
\label{eqn:spec_dist}
d(A,B)= 
\max_{u\in \op{V}_h}\left|\ln \innprd{A u}{u} - \ln \innprd{B u}{u} \right | =
\max \{|\ln \lambda|\ :\ \lambda\in \sigma(A,B)\}.
\end{eqnarray}
For the optimal control of elliptic PDEs (the case $s=1$), and under maximum regularity assumptions, it is known that
\begin{equation}
\label{eq:ellipticopt}
d(\op{H}_h^1,\op{G}_h^1) \le C \frac{h^2}{\beta}.
\end{equation}
Consequently, when solving~\eqref{equ:discretesys} using multigrid preconditioned conjugate gradient (CG), the number of iterations will decrease
with increasing resolution at the optimal rate. This is significant, since at higher resolutions the most expensive operation
is precisely the Hessian-vector multiplication. A decrease in the power of $h$ in~\eqref{eq:ellipticopt}, which can occur in a number of
instances (boundary control, loss of elliptic regularity), results in~\eqref{equ:twogridprec} becoming a less efficient preconditioner.

We conduct our analysis using Lemma 1 in~\cite{Barker_Draganescu_2020}, which requires estimating the operator $L^2(\Omega)$-norm 
\begin{equation}
\label{eq:l2opnorm}
\|\op{K}_h^s-\embed\op{K}_{2h}^s\pi_{2h}\| = \sup_{z\in \op{V}_h}\frac{ \|(\op{K}_h^s-\embed\op{K}_{2h}^s\pi_{2h})z\|}{\|z\|},
\end{equation}
where $\|\cdot\|$ on the right-hand-side denotes the norm in $L^2(\Omega)$, and $\embed$ also denotes the restriction of $\embed$ to $\op{V}^0_h$. 
From here on, $\|\cdot\|$ without subscripts represents either the vector or the operator $L^2$-norm, depending on the context. 
{Notice that} only the control-to-state solution operators play a role in~\eqref{eq:l2opnorm}.
The estimation process is based on the following apriori estimate in Corollary 2 
from~\cite{SDohr_CKahle_SRogovs_PSwierczynski_2019a}, which assumes $\Omega$ to be convex in $\R^2$ or $\R^3$: for any $s\in (0,1)$ and $\varepsilon'>0$, there exists $C=C(\varepsilon',s)$ 
so that
\begin{equation}
\label{eq:l2opnormdisc}
 \|\op{K}_h^s z -\op{K}^s z\| \le C h^{2 s-\varepsilon'} \|z\|.
\end{equation}
We also recall the following regularity estimate: for $s\in (0,1)$ there
exists $C$ (uniformly bounded in $s$) so that:
\begin{equation}
\label{eq:regularity}
 \|\op{K}^s z\|_{\mathbb{H}^{2s}(\Omega)} \le C \|z\|.
\end{equation}
This immediately follows: if $z = \sum_{k=1}^{\infty} z_k \varphi_k$ in $L^2(\Omega)$, and
    $u = \sum_{k=1}^\infty u_k \varphi_k = \mathcal{K}^s z$, solves the state equation \eqref{eq:state}, then
  \[
  u_k = \lambda_k^{-s} z_k .
  \]	
  Then from the definition of $\mathbb{H}^s$-norm in \eqref{eq:Hs}, we have that 
	\[
		\| u \|_{\mathbb{H}^{2s}(\Omega)}^2
		= \sum_{k=1}^{\infty} \lambda_k^{2s} u_k^2 \varphi_k 
		= \sum_{k=1}^{\infty} \lambda_k^{2s} \lambda_k^{-2s} z_k^2 \varphi_k 
		= \|z\|^2 .
	\]
As a consequence of convergence~\eqref{eq:l2opnormdisc} and regularity~\eqref{eq:regularity} we obtain the following uniform bound (with respect to $h$)
of the operator norm of $\op{K}_h^s$: there exists $L_s$ independent of $h$ so that 
\begin{equation}
\label{eq:unifbound}
\|\op{K}^s_h z\|\le L_s \|z\|,\ \ \forall z\in \op{V}_h.
\end{equation}
\begin{lemma}
\label{lemma:twolevelapprox}
Assume $\Omega\subset \R^N$ with $N=2, 3$ {be a convex polygonal} bounded domain. Then for any $\varepsilon'>0$ and $s\in (0,1)$ {there is a constant} $C_s>0$ so that
\begin{equation}
\label{eq:ahest}
\|\op{K}_h^s-\embed\op{K}_{2h}^s\pi_{2h}\| \le C_s h^{2 s-\varepsilon'}.
\end{equation}
\end{lemma}
\begin{proof}
For $z\in \op{V}_h$ we have 
\begin{align}
\label{eq:ah_estimate}
\|(\op{K}_h^s-\embed\op{K}_{2h}^s\pi_{2h})z\|  &\le   
\|(\op{K}_h^s-\op{K}^s)z\|  \nonumber \\
&\quad+\|(\op{K}^s(I-\pi_{2h})z\| +
\|(\op{K}^s-\op{K}_{2h}^s)\pi_{2h})z\|,
\end{align}
where we omitted the embedding operators. Using~\eqref{eq:l2opnormdisc} we can bound the first and third terms on the right-hand side 
of~\eqref{eq:ah_estimate} by 
\begin{eqnarray}
\label{eq:firstlastterm}
\|(\op{K}_h^s-\op{K}^s)z\| \le C h^{2 s -\varepsilon'}\|z\|,\ \ \ \|(\op{K}^s-\op{K}_{2h}^s)\pi_{2h} z\| \le  C (2h)^{2 s -\varepsilon'}\|z\| ,
\end{eqnarray}
where we have also used $\|\pi_{2h}z\|\le \|z\|$. 
For the middle term {in \eqref{eq:ah_estimate}} we interpolate between the inequalities (see~\cite{SCBrenner_RLScott_2008a})
\begin{equation}
\label{eq:ineqprojh_l2h1}
\|u-\pi_{h}u\|\le C h \|u\|_{H^1_0(\Omega)},\ \ \|u-\pi_{h}u\|\le \|u\|_{L^2(\Omega)},
\end{equation}
that hold for all $u\in H^1_0(\Omega)$, respectively $u\in L^2(\Omega)$. It follows that
\begin{equation}
\label{eq:ineqprojhs}
\|u-\pi_{h}u\|\le C h^s \|u\|_{\mathbb{H}^s(\Omega)},\ \ \forall u\in \mathbb{H}^s(\Omega).
\end{equation}
Hence,
\begin{eqnarray}
\nonumber
\|\op{K}^s(I-\pi_{2h})z\| &= & \sup_{v\in L^2(\Omega)} \frac{|\innprd{\op{K}^s(I-\pi_{2h})z}{v}|}{\|v\|} = 
\sup_{v\in L^2(\Omega)} \frac{|\innprd{(I-\pi_{2h})z}{\op{K}^s v}|}{\|v\|} \\
\nonumber
& = & \sup_{v\in L^2(\Omega)} \frac{|\innprd{(I-\pi_{2h})z}{\op{K}^s v - \pi_{2h}\op{K}^sv}|}{\|v\|}\\
\nonumber
&\stackrel{\eqref{eq:ineqprojhs}}{\le}&
 (2h)^{2s}\sup_{v\in L^2(\Omega)} \frac{\|(I-\pi_{2h})z\| \|\op{K}^s v\|_{\mathbb{H}^{2s}(\Omega)}}{\|v\|} \\
\label{eq:middleterm}
& \stackrel{\eqref{eq:regularity}}{\le}&  C (2h)^{2s} \|(I-\pi_{2h})z\| \le C' h^{2s} \|z\|.
\end{eqnarray}
The conclusion follows from~\eqref{eq:firstlastterm} and~\eqref{eq:middleterm}. 
\end{proof}

The next theorem follows from Lemma~\ref{lemma:twolevelapprox} and Lemma 1 in~\cite{Barker_Draganescu_2020}.
\begin{theorem}
\label{thm:twolevelapprox}
If $C_s h^{2 s-\varepsilon'} \le \beta/ (4 L_s)$, then
\begin{equation}
\label{eq:twolevelprecest}
d(\op{H}_h^s,\op{G}_h^s) \le 4 L_s\beta^{-1} h^{2 s-\varepsilon'}.
\end{equation}
\end{theorem}
This result certifies that the quality of the two-grid (and hence multigrid) preconditioner is improving with
increasing resolution, as in the elliptic case, but at a rate that is degrading as $s$ decreases to $0$. Consequently, the preconditioner
is expected to be less efficient as $s$ decreases. 
At the same time, the coarsest mesh that can be used may also need to be finer and finer as $s$ decreases due to the hypothesis in
Theorem~\ref{thm:twolevelapprox}; hence, the number of levels that can be used {at some point will necessarily} be smaller.
Remarkably, the numerical results in Section~\ref{sec:numerics} show an improved picture: they suggest that in fact a 
significantly stronger estimate holds. Hence, we formulate the following conjecture.

\begin{conjecture}
\label{conj:distest}
Assuming the {domain} is convex, there is a constant $\tilde{C}_s$ independent of $h$, so that, if $h$ is sufficiently small,
\begin{equation}
\label{eq:distest}
d(\op{H}_h^s,\op{G}_h^s) \le 
	\left\{ 
	\begin{array}{ll}
		\tilde{C}_s \beta^{-1}h^{4s} \, , 	& \mbox{if } 0<s < 1/2	\\
		\tilde{C}_s \beta^{-1}h^{2} \, , & \mbox{if } 1/2 \le s < 1 \, .
	\end{array}	
	\right.	
\end{equation}
\end{conjecture}
It is notable that~\eqref{eq:distest} is consistent with the classical result~\eqref{eq:ellipticopt} for $s=1$, and also with 
Theorem~\ref{thm:twolevelapprox} as $s\approx 1$. However, it shows that the preconditioner is uniformly very good when $1/2\le s < 1$ and 
even with the classical case $s=1$, and is twice as efficient compared to what the analysis predicts for $0<s<1/2$.
{Proving} Conjecture~\ref{conj:distest} requires a different approach from proving Theorem~\ref{thm:twolevelapprox}, since we do not 
expect any superconvergence to hold in~\eqref{eq:ahest}. Instead, we expect the proof of the conjecture to involve higher order estimates 
in weaker norms for the control-to-state map, in addition to more refined regularity results.

\section{Numerical experiments}
\label{sec:numerics}
We {have} performed two kinds of numerical experiments. First we aim to verify~\eqref{eq:distest} directly by building matrices corresponding to 
the Hessian and the two-grid preconditioner for a set of grids with $h_j=2^{-j}$, $j=j_{\min}, \dots, j_{\max}$, followed by a direct computation of
$d^s_{h_j} := d(\op{H}_{h_j}^s,\op{G}_{h_j}^s)$ using generalized eigenvalues ($\op{H}_{h_j}^s u=\lambda  \op{G}_{h_j}^s u$).
Then we form the ratios $d^s_{h_{j-1}}/d^s_{h_j}$ to confirm the formula~\eqref{eq:distest}. 
We show results for $\beta=1$ and $\beta=0.1$ for this purpose. However, due to the sizes of the matrices involved, these computations
are limited to the one-dimensional case $\Omega=(0,1)$.
The results for $s=0.25, 0.3, 0.4, 0.5, 0.6, 0.7$ are shown in Table~\ref{t:rate}, and they strongly support Conjecture~\ref{conj:distest}.
The precise values of $j_{\min}, j_{\max}$ vary with~$s$, due primarily to memory limitation (smaller $s$ requires more memory).
It is notable that the spectral distances in the lower part of the table are approximately ten times larger than 
{their} counterparts in the upper half (for 
a value of $\beta$ that is ten times smaller), thus also supporting the dependence on $\beta$ in~\eqref{eq:distest}.

\begin{table}[tbhp]
  {\footnotesize
    \caption{Direct measurements of spectral distances in one spatial dimension for two different regularization parameters $\beta$
    		and different mesh sizes. The dependence on $\beta$ is according to the predicted theory in Theorem~\ref{thm:twolevelapprox}
		and the dependence on meshsize $h$ is according to the Conjecture~\ref{conj:distest}.}
    \label{t:rate}
    \begin{center}
      \begin{tabular}{|l|l|l|l|l|l|l|l|} \hline
        \multicolumn{8}{|c|}{$\boldsymbol\beta=1$}\\\hline
        $N$   & 16 & 32 & 64 & 128 & 256 & 512 & 1024\\\hline
        $s=0.25$ & 3.51e-2 &  1.78e-2 &  8.97e-03 &  4.50e-3 &  2.25e-3  & 1.13e-3 &  5.64e-4\\
        $\log_2(d_i/d_{i+1})$ & 0.9771 &   0.9910  &  0.9961  &  0.9982  &  0.9991  &  0.9996 & \\\hline
        $s=0.3$ & 1.82e-2 &  8.02e-3 &  3.51e-3 &  1.53e-3 &  6.66e-4 &  2.90e-4 &  1.26e-4\\
        $\log_2(d_i/d_{i+1})$ & 1.1807  &  1.1931  &  1.1976  &  1.1991  &  1.1997  &  1.1999 & \\\hline
        $s=0.4$ & 4.81e-3 &  1.61e-3 &  5.34e-4 &  1.76e-4 &  5.82e-5 &  1.92e-5 &  6.34e-6\\
        $\log_2(d_i/d_{i+1})$ & 1.5780 &   1.5934  &  1.5976 &   1.5993 &   1.5998 &   1.5999& \\\hline\hline
        $N$   & 64 & 128 & 256 & 512 & 1024 & 2048 & 4096\\\hline
        $s=0.5$ & 1.20e-4  & 2.71e-5 &  6.16e-6 &  1.46e-6  & 3.43e-7  & 8.24e-8 &  2.03e-08\\
        $\log_2(d_i/d_{i+1})$ & 2.1432 &   2.1386 &   2.0742 &   2.0949 &   2.0566 &   2.0194 & \\\hline
        $s=0.6$ &  8.53e-5 &  1.89e-5 &  4.45e-6 &  1.02e-6 &  2.40e-7 &  5.81e-8 &  1.40e-8\\
        $\log_2(d_i/d_{i+1})$ & 2.1730 &   2.0865 &   2.1167 &   2.0946 &   2.0486  &  2.0536 & \\\hline
        $s=0.7$ & 5.83e-5 &  1.37e-5 &  3.15e-6 &  7.25e-7 &  1.71e-7 &  4.13e-8 &  1.01e-8\\
        $\log_2(d_i/d_{i+1})$ & 2.0930 &   2.1184 &   2.1195 &   2.0805 &   2.0524 &   2.0350 &\\\hline\hline
        \multicolumn{8}{|c|}{$\boldsymbol\beta=0.1$}\\\hline
        $N$   & 16 & 32 & 64 & 128 & 256 & 512 & 1024\\\hline
        $s=0.25$ & 3.06e-1 &  1.66e-1 & 8.63e-2 & 4.41e-2 & 2.23e-2 &  1.12e-2 &  5.62e-3 \\
        $\log_2(d_i/d_{i+1})$   & 0.8846 & 0.9391 & 0.9685 & 0.9840 &  0.9919 &  0.9959 &  \\\hline
        $s=0.3$  & 1.68e-1 & 7.75-2 &  3.45e-2 & 1.52e-2 &  6.64e-3 & 2.90e-3 &  1.26.e-3 \\
        $\log_2(d_i/d_{i+1})$   & 1.1210 & 1.1651 & 1.1850 & 1.1935 &  1.1972  &  1.1988 & \\\hline
        $s=0.4$ &  4.71e-2  & 1.60e-2 & 5.33e-3 &  1.76e-3 &  5.82-4 &  1.92e-4 & 6.34e-5 \\
        $\log_2(d_i/d_{i+1})$ & 1.5578 &  1.5864 &  1.5953 & 1.5985 &  1.5995 &  1.5998 & \\\hline\hline
        $N$   & 64 & 128 & 256 & 512 & 1024 & 2048 & 4096\\\hline
        $s=0.5$ &  8.74e-4 &  2.12e-4 &  5.20e-5 &  1.29e-5 &  3.20e-6 & 7.98e-7 & 1.99e-8 \\
        $\log_2(d_i/d_{i+1})$   & 2.0416 & 2.0299 & 2.0118 & 2.0111 & 2.0045 & 2.0019 & \\\hline
        $s=0.6$ &  5.72e-4 &  1.28e-4 & 3.02e-5 & 7.01e-6 & 1.65e-6 & 4.00e-7 & 9.73e-8 \\
        $\log_2(d_i/d_{i+1})$  & 2.1625 &  2.0807 & 2.1077 & 2.0866 &  2.0445 &  2.0400 & \\\hline
        $s=0.7$  &  4.40e-4  & 1.03e-4 &  2.39e-5 & 5.51e-6  & 1.31e-6 &  3.15e-7 & 7.72e-8 \\
        $\log_2(d_i/d_{i+1})$  & 2.0900 &  2.1138 &    2.1141 &  2.0764 &  2.0496 &  2.0330 & \\\hline
        \end{tabular}
    \end{center}
  }
\end{table}

The second kind of numerical results are actual two-dimensional solves in $\Omega=(0,1)^2$ of~\eqref{equ:discretesys}, i.e.,
our optimal control problem with a multigrid version of the preconditioner. The data is 
$u_d(x,y)=\sin(4\pi x) \sin(3 \pi y)$.
For each case considered, we compare the
number of unpreconditioned CG iterations to the number of multigrid preconditioned CG (MGCG) iterations, and we report the
wall-clock times. The results are reported in Table~~\ref{tab:2D_runswide} and also propagated in Figures~\ref{f:iterations} 
and \ref{f:clock_times}. The solvers are all matrix-based, in the sense that the sparse 
matrices implementing the operators $\op{K}_h^s$ are formed in block-diagonal form and prefactored. Only the coarsest Hessian
is formed at resolution $32 \times 32$, which is used as the base case for all cases considered. 
The effect of decreasing the value of the regularizer $\beta$ and/or that of the parameter $s$  is
an increase in the number of CG iterations. {In order} to maintain the number of unpreconditioned CG iterations
between $20$ and $50$ (for {illustration} purposes) we have chosen slightly larger values of $\beta$ as we decreased $s$ in the experiments 
described below. 
The number of CG iterations also indicates the difficulty of the
problem at hand, as it corresponds to the number of relevant eigenmodes that can be recovered for the control for a given
problem setting. All computations were performed using \textsc{Matlab}  on a system with two eight-core 2.9 GHz Intel 
Xeon E5-2690 CPUs and 256 GB memory.

The cases include $s=0.25, 0.3, 0.4, 0.5, 0.6, 0.7$. The results show a dramatic reduction in the 
number of MGCG iterations compared to unpreconditioned CG, as well as a reduction in computing time. 
It is notable that for each case, the number of MGCG iterations is ultimately decreasing with increasing resolution. E.g., for
$s=0.4$ the number of multigrid CG iterations, decreases from 7 on a $64\times 64$ grid to 3 on a $512\times 512$ grid, while
the number of unpreconditioned CG iterations is virtually constant. However, for $s=0.25$ the decrease is less dramatic.
It is expected 
that for a regular, iterative or parallel implementation of the matrix-vector product of the Hessian,
the dramatic decrease in number of iterations {will} be reflected in the decrease of computing time, since
the most expensive iteration remains at the finest-level fractional Poisson solve.

\begin{table}[tbhp]
  {\footnotesize
    \caption{Iteration counts for unpreconditioned CG vs. MGCG with base case $32\times 32$; wall-clock times 
    are shown in seconds in parenthesis. The time marked with ${}^*$ is not relevant, because the computation was forced into much 
    slower swap space due to memory limitations.}
    \label{tab:2D_runswide}
    \begin{center}
      \begin{tabular}{|l|l|l|l|l|l|l|l|} \hline
        \multicolumn{2}{|c|}{$64\times 64$} & \multicolumn{2}{c|}{$128\times  128$} & 
        \multicolumn{2}{c|}{$256\times 256$}& \multicolumn{2}{c|}{$512\times 512$} \\\hline
        CG & MGCG & CG & MGCG & CG & MGCG & CG & MGCG\\\hline
        \multicolumn{8}{|c|}{$s=0.25,\ \ \beta=10^{-2}$}\\\hline
         24 (25) &  9 (10)   & 25  (198)  &  19 (174)   &  25 (1482)  &  12 (880)   & 25 (11326)  & 8 (23876*)   \\\hline
        \multicolumn{8}{|c|}{$s=0.3,\ \ \beta=10^{-2}$}\\\hline
        21 (19) &  6 (6)   &  21 (148)  &  6 (54)   & 21 (1110)  &  4 (299)   & 21 (8805)  & 3 (2045)   \\\hline
        \multicolumn{8}{|c|}{$s=0.4,\ \ \beta=10^{-3}$}\\\hline
        35 (28)  &  7 (6) & 35 (206)  &  7 (52)  & 36 (1635)  &  4 (261)   & 35 (11753) & 3 (1585)     \\\hline
        \multicolumn{8}{|c|}{$s=0.5,\ \ \beta=10^{-3}$}\\\hline
        23 (18)  & 5 (4.5)  & 23 (136) & 4 (32)    & 23 (1015)  & 3 (200)    & 23 (7743) & 2  (1124)    \\\hline
        \multicolumn{8}{|c|}{$s=0.5,\ \ \beta=10^{-4}$}\\\hline
        57 (43)  & 8  (7)  & 56  (316) &  11 (76)    & 55 (2369)   &  7 (402)   & 55 (18090)  &  4 (1900)   \\\hline
        \multicolumn{8}{|c|}{$s=0.6,\ \ \beta=10^{-4}$}\\\hline
        38 (30)  &  6  (5)  & 38 (226) & 4 (33)  &  37 (1687)  &  3 (210)   & 37  (12591)  &  3 (1635)   \\\hline
        \multicolumn{8}{|c|}{$s=0.7,\ \ \beta=10^{-4}$}\\\hline
         27 (25)  &  5  (5)  & 26 (183) & 4 (38)  &  25 (1315)  & 3  (240)   & 25  (9991)  & 3  (2103)   \\\hline
      \end{tabular}
    \end{center}
}
\end{table}

\begin{figure}[htb]
	\includegraphics[width=\textwidth]{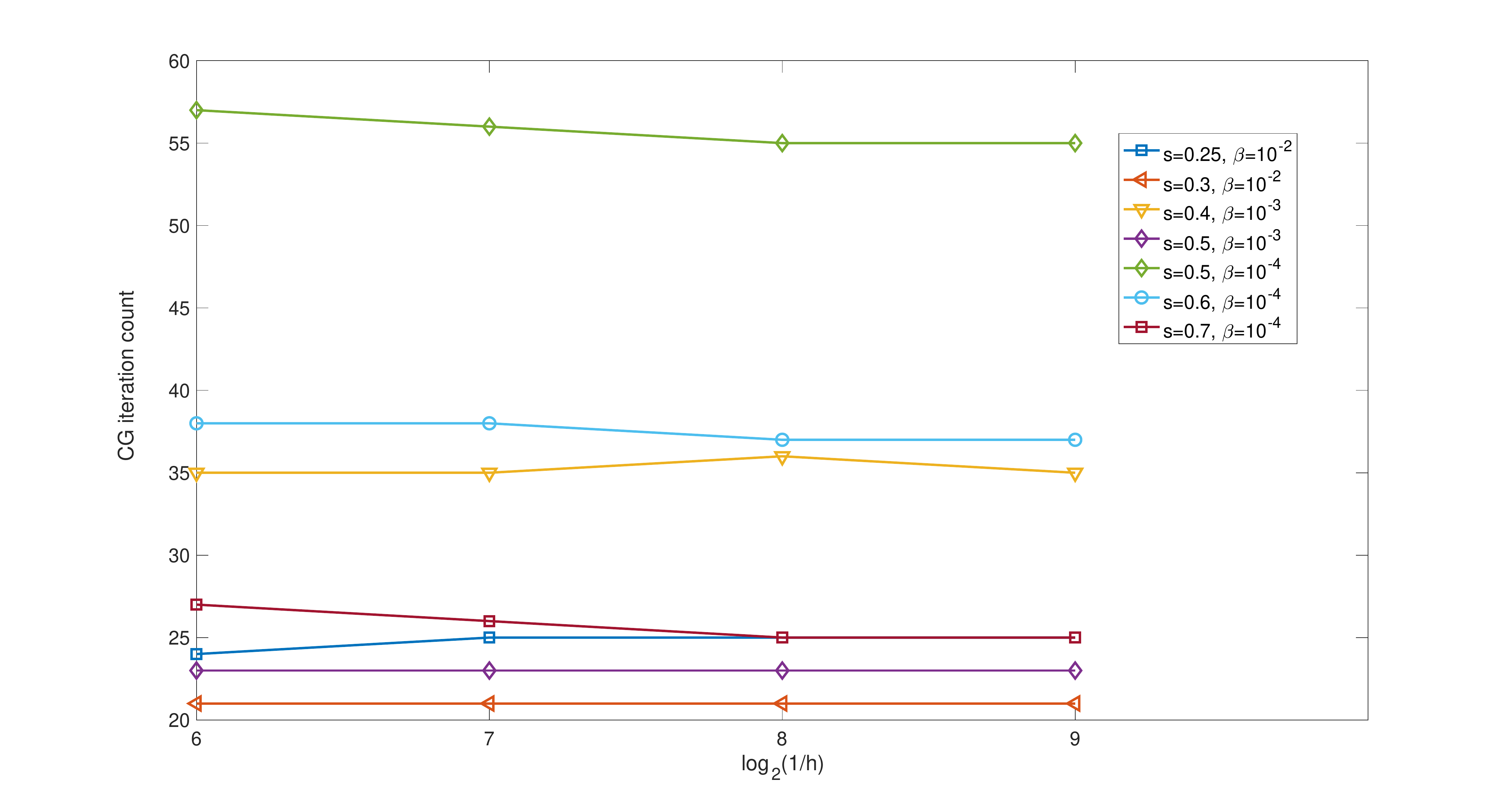}
	\includegraphics[width=\textwidth]{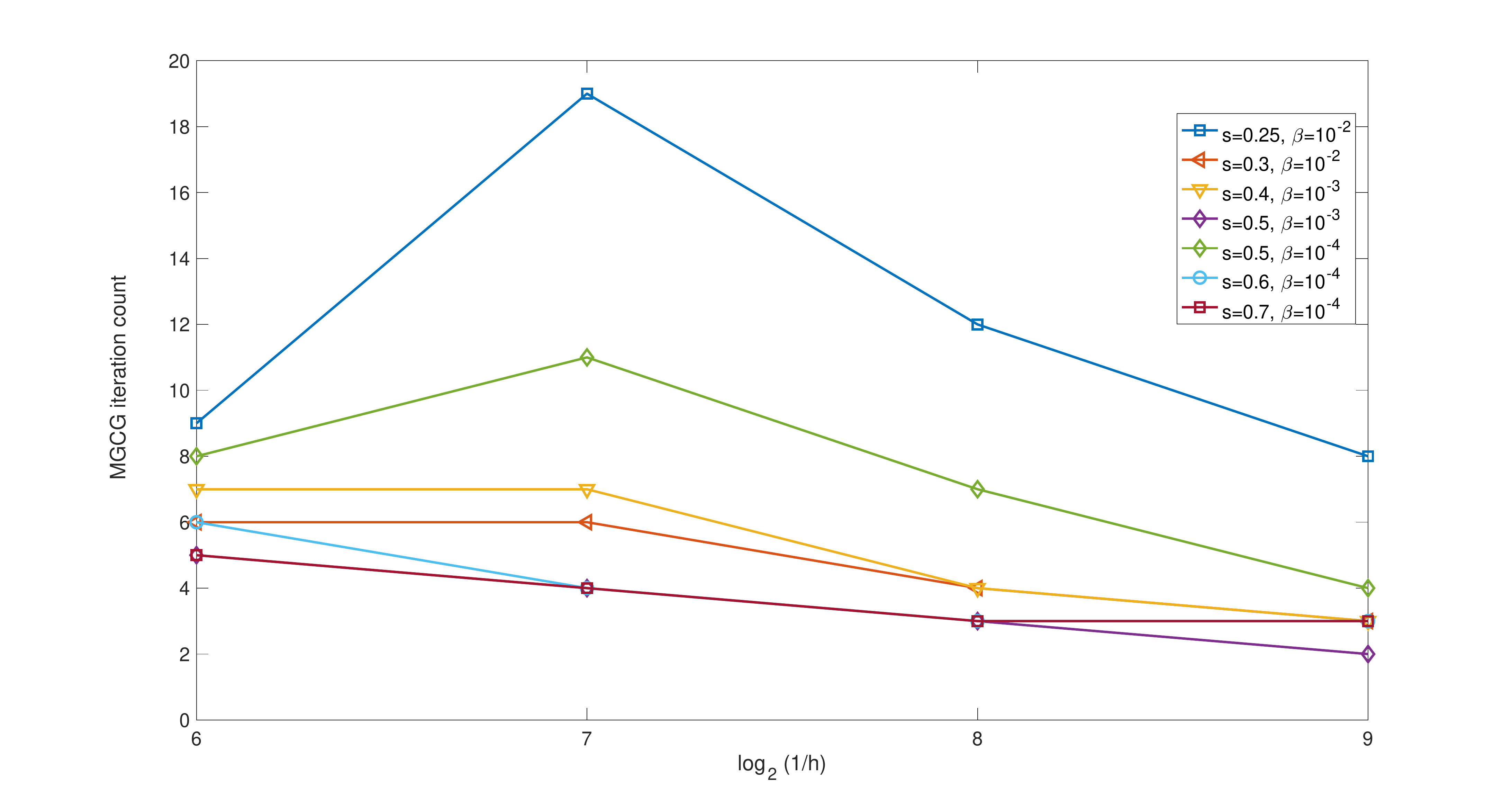}	
	\caption{The top  and bottom panels respectively show the iteration counts
	for CG and MGCG with respect to the mesh-size. As expected, we observe the 
	iteration count to be independent of the mesh-size for CG and this iteration count
	decreases in case of MGCG as the mesh-size decreases, 
        except for, perhaps, the transition from two-to three grids; since the base case is $h=2^{-5}$, we 
        see in some cases a slight increase in MGCG iteration MGCG from $h=2^{-6}$ to $h=2^{-7}$.
        }
	\label{f:iterations}
\end{figure}

\begin{figure}[htb]
	\includegraphics[width=\textwidth]{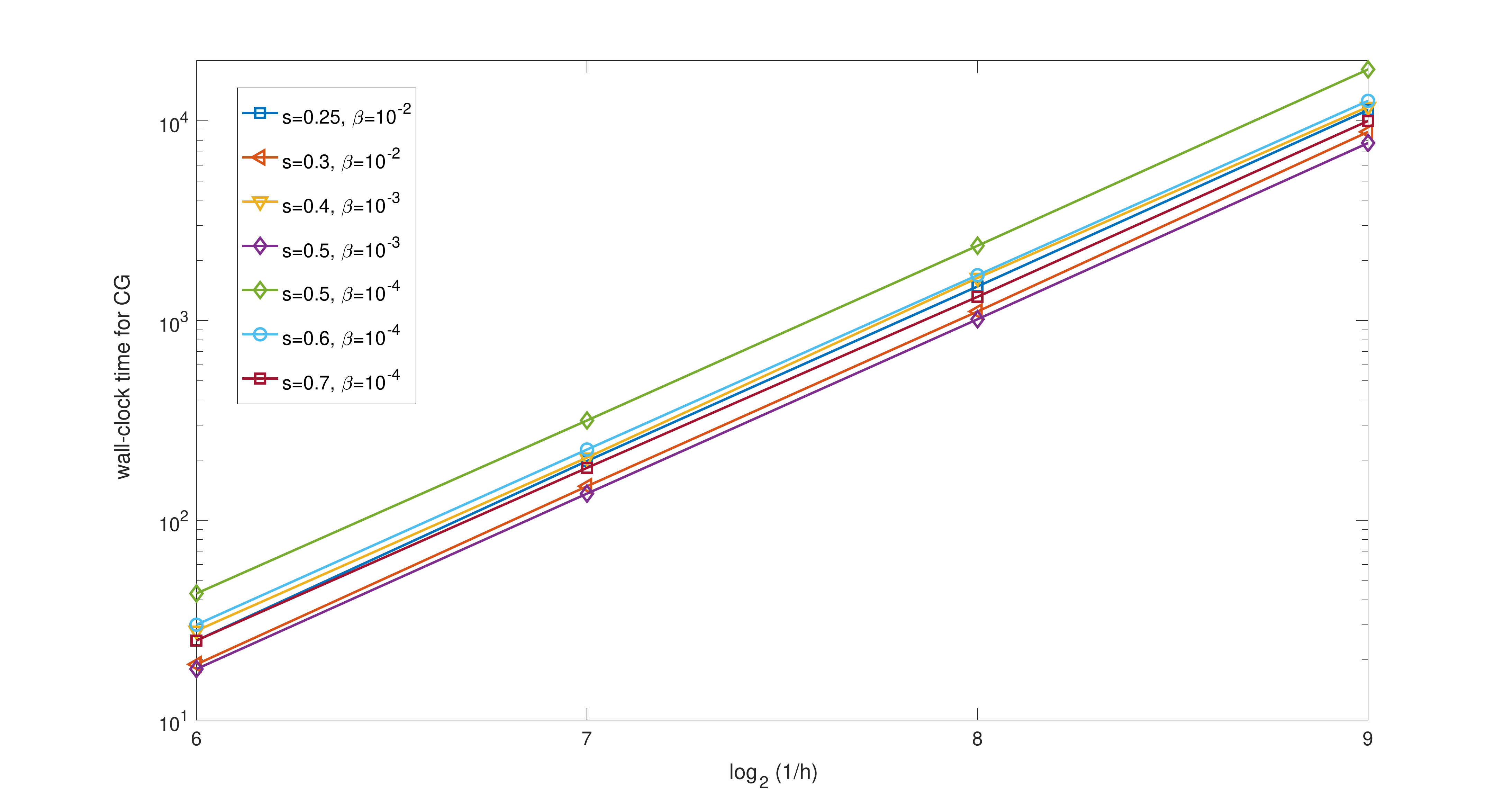}
	\includegraphics[width=\textwidth]{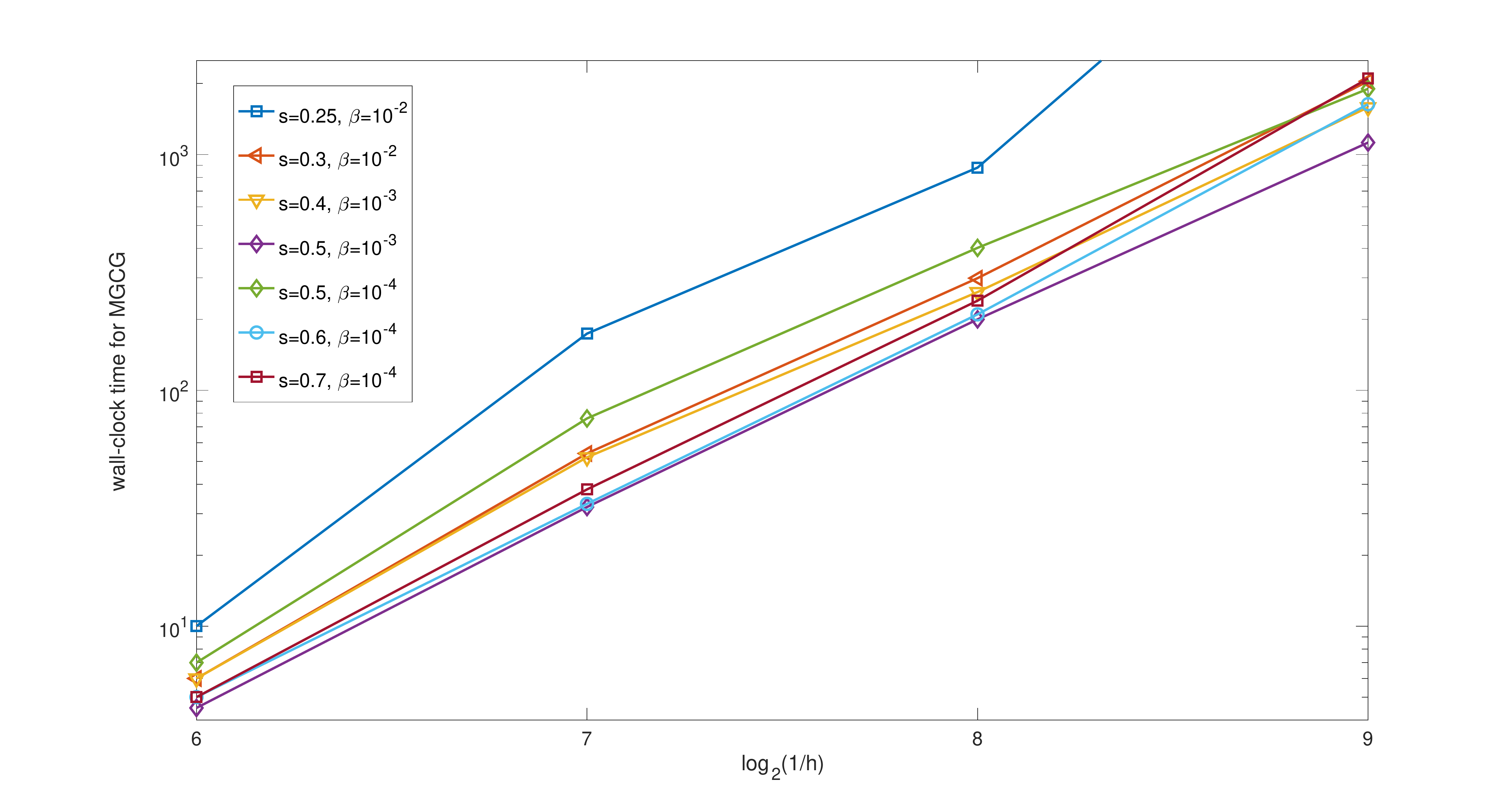}		
	\caption{The top  and bottom panels respectively show the wall clock time 
	for CG and MGCG with respect to the mesh-size. As expected, on the log-log plot 
	we observe a perfectly linear behavior for CG and a sublinear behavior for MGCG
        (we excluded from range the value marked with * in Table~\ref{tab:2D_runswide}, since it only reflects that 
        the machine was forced into much slower swap space).}
	\label{f:clock_times}
\end{figure}

\bibliographystyle{elsarticle-num} 
\bibliography{references}

\end{document}